\documentclass[12pt,a4paper]{amsart}

\usepackage{enumerate}
\usepackage{hyperref}

\usepackage{tikz-cd}
\pgfqkeys{/tikz/commutative diagrams}{
  equal/.style={
    /tikz/arrows=-,
    /tikz/commutative diagrams/double line}
}
\usetikzlibrary{external}
\usepackage{a4wide}

\theoremstyle{plain}
\newtheorem*{thm*}{Theorem}
\newtheorem{thm}{Theorem}[section]
\newtheorem{prop}[thm]{Proposition}
\newtheorem{lem}[thm]{Lemma}

\theoremstyle{definition}
\newtheorem{defn}[thm]{Definition}
\newtheorem{ex}[thm]{Example}
\newtheorem{rmk}[thm]{Remark}
\newtheorem*{q}{Question}

\numberwithin{equation}{thm}

\newcommand{\emphbf}[1]{\emph{\textbf{#1}}}

\newcommand{\rad}[1]{\radoperator(#1)}
\newcommand{\roof}[5]{%
\begin{tikzcd}[ampersand replacement=\&, row sep=0pt]
\& {#3} \ar{ld}[swap]{#2} \ar{rd}{#4} \\
{#1} \&\& {#5}
\end{tikzcd}
}

\DeclareMathOperator{\radoperator}{rad}
\DeclareMathOperator{\id}{id}
\DeclareMathOperator{\Hom}{Hom}
\DeclareMathOperator{\Ext}{Ext}
\DeclareMathOperator{\modu}{mod}
\DeclareMathOperator{\End}{End}
\DeclareMathOperator{\charac}{char}
\DeclareMathOperator{\Coker}{Coker}
\DeclareMathOperator{\maxSpec}{maxSpec}
\DeclareMathOperator{\add}{add}
\DeclareMathOperator{\HH}{HH}
\DeclareMathOperator{\proj}{proj}
\DeclareMathOperator{\injdim}{injdim}
\DeclareMathOperator{\projdim}{projdim}
\DeclareMathOperator{\op}{op}
\DeclareMathOperator{\Ann}{Ann}
\DeclareMathOperator{\Homcom}{\mathcal{H}om}
\DeclareMathOperator{\ev}{ev}
\DeclareMathOperator{\comp}{comp}

\DeclareMathOperator{\Der}{\mathcal{D}}
\DeclareMathOperator{\Kom}{\mathcal{K}}

\begin{document}

\title{Derived invariance of support varieties}
\author{Julian K\"ulshammer}
\author{Chrysostomos Psaroudakis}
\author{\O ystein Skarts\ae terhagen}
\date{\today}
\thanks{The authors want to thank Hongxing Chen, Steffen Koenig and \O yvind Solberg for useful discussions and valuable comments. These thanks are extended to Wei Hu for pointing out Remark \ref{standardnotneeded}}

\address{Julian K\"ulshammer\\
Institute of Algebra and Number Theory,
University of Stuttgart \\ Pfaffenwaldring 57 \\ 70569 Stuttgart,
Germany} \email{kuelsha@mathematik.uni-stuttgart.de}

\address{Chrysostomos Psaroudakis\\
Department of Mathematical Sciences, Norwegian University of Science and Technology \\
 7491 Trondheim, Norway
} \email{chrysostomos.psaroudakis@math.ntnu.no}

\address{\O ystein Skarts\ae terhagen\\
Department of Mathematical Sciences, Norwegian University of Science and Technology \\
 7491 Trondheim, Norway
} \email{oystein.skartsaterhagen@math.ntnu.no}



\begin{abstract}
The (Fg) condition on Hochschild cohomology as well as the support variety theory are shown to be invariant under derived equivalence.
\end{abstract}

\maketitle

\tikzexternaldisable
\section{Introduction}

In 1983, Carlson \cite{Car83} introduced the concept of the support
variety of a module for group algebras using ordinary cohomology. This
concept and its generalisations to other classes of algebras led to
several interesting applications including for example a criterion for
the representation type of an algebra \cite{Fa07}, and connections to
vector bundles \cite{FP11}.  In 2004, Snashall and Solberg \cite{SS04}
extended the theory to arbitrary finite dimensional algebras using
Hochschild cohomology, and together with Erdmann, Holloway and
Taillefer \cite{EHSST04} they introduced certain finiteness
assumptions in order to obtain a theory of support varieties with good
properties. These finiteness assumptions are called the (Fg)
condition. The motivation for this article is the following question.

\begin{q}
Which types of equivalences between two algebras preserve the (Fg) condition, and more generally support varieties?
\end{q}

One answer to this question was given in \cite[Theorem 4.1]{Lin11}. Therein, Linckelmann has shown that the (Fg) condition is stable under separable equivalence provided both algebras are symmetric. In particular this includes the cases of Morita equivalences, derived equivalences, and stable equivalences of Morita type (in the case of symmetric algebras). It seems to be well-known to the experts that the (Fg) condition is stable under Morita equivalence but the authors weren't able to find a reference for this fact (which will also follow from our main theorem). 

Since the support variety of any bounded complex of projectives vanishes, a natural guess would also be that an equivalence of the singularity categories $\Der^b(A)/\Kom^b(\proj A)$ yields invariance of support varieties. However, in joint work of the two last named authors with \O yvind Solberg \cite[Example 5.5]{PSS}, a counterexample is provided. 
The last named author of this article has shown that a singular equivalence of Morita type (with level) preserves (Fg) provided both algebras are known to be Gorenstein~\cite{Ska16}. This also includes the case of stable equivalence of Morita type for self-injective algebras. Here we show the following:

\begin{thm*}
Let $A$ and $B$ be derived equivalent $k$-algebras. Then the following hold. 
\begin{enumerate}[(a)]
\item The (Fg) condition holds for $A$ if and only if it holds for $B$. 
\item If $F\colon \Der^b(A)\to \Der^b(B)$ is a derived equivalence, then the support varieties of $M^*$ and $F(M^*)$ are isomorphic for every bounded complex of $A$-modules $M^*$.
\end{enumerate}
\end{thm*}

Using the result in \cite{Ska16} part (a) can also be obtained as follows: The first step is to show that if $A$ satisfies (Fg), then $A$ and $B$ are both Gorenstein. The algebra $A$ is Gorenstein by \cite[Proposition 1.2]{EHSST04}. The property of algebras being Gorenstein is preserved under derived equivalence by combining results of Happel \cite[Lemma~1.5]{Hap91} and Rickard \cite[Theorem~6.4, Proposition~9.1]{Ric89}. Hence, $B$ is also Gorenstein. 
Part (a) of the Theorem now follows by using that any derived equivalence induces a singular equivalence of Morita type with level \cite[Theorem~2.3]{Wan14} and that a singular equivalence of Morita type with level preserves (Fg) provided both algebras are Gorenstein \cite{Ska16}.

The structure of the article is as follows. Section~\ref{derivedequivalences} recalls the relevant parts of the theory of derived equivalences. Section~\ref{supportvarieties} recalls the definition of a support variety for complexes and proves the Theorem. In the last section we give an example illustrating that our result can be applied to show (Fg) for an algebra where previous criteria did not work.
\smallskip

\textbf{Conventions and Notation:} Throughout the article let $k$ be a field. We assume all $k$-algebras to be finite dimensional except for the subalgebras of the Hochschild cohomology ring. For a $k$-algebra $A$ we write  $\modu A$ for the category of finitely generated left $A$-modules, and $\proj A$ for the subcategory consisting of projective modules. We write $\Der^b(A)$ for the bounded derived category of $\modu A$ and $\Der^-(A)$ for the derived category of bounded above complexes. We denote the homotopy category of bounded above complexes of projectives by $\Kom^-(\proj A)$ and its subcategory of complexes with bounded homology by $\Kom^{-,b}(\proj A)$. There is a triangle equivalence $N_A\colon \Kom^-(\proj A)\to \Der^-(A)$ which restricts to an equivalence $N_A\colon \Kom^{-,b}(\proj A)\to \Der^b(A)$. For complexes $X^*$ and $Y^*$ of $A$-modules we define $\Hom^*_{\Der^b(A)}(X^*,Y^*):=\bigoplus_{n\in \mathbb{Z}}\Hom_{\Der^b(A)}(X^*,Y^*[n])$ and $\End^*_{\Der^b(A)}(X^*):=\Hom^*_{\Der^b(A)}(X^*,X^*)$.
Furthermore we denote by $A^e:=A\otimes_k A^{\op}$ the enveloping algebra of $A$. The Hochschild cohomology ring is defined as $\HH^*(A):=\End^*_{\Der^b(A^e)}(A)$. If $\charac k\neq 2$, let $\HH^{\ev}(A):=\HH^{2*}(A)$, if $\charac k=2$, let $\HH^{\ev}(A):=\HH^*(A)$.

\section{Derived equivalences}\label{derivedequivalences}

In this section, we state the definitions and results we need regarding derived categories and derived equivalences.  We recall Rickard's derived analogue of Morita's theorem and as an application the invariance of Hochschild cohomology under derived equivalences.  These results were first proved in \cite{Ric89} and \cite{Ric91}.  For an overview of Morita theory for derived categories, see the books \cite{Zim14} and \cite{KZ98}.  Furthermore, we show that in our setting, the derived tensor product is associative.  This is a fact that should be well-known, but for which we could not find a precise reference in the literature.

\begin{defn}
Two $k$-algebras $A$ and $B$ are \emphbf{derived equivalent} if there exists a triangle equivalence $F\colon \Der^b(A)\to \Der^b(B)$.  The functor $F$ is then a \emphbf{derived equivalence} between $A$ and $B$, and it is \emphbf{of standard type} if $F\cong X^*\otimes_A^{\mathbb{L}}-$ for a complex $X^*$ of $A$-$B$-bimodules.
\end{defn}

The following is part of Rickard's celebrated version of Morita's theorem:

\begin{thm}\label{rickardstheorem}
Let $A$ and $B$ be $k$-algebras. Let $F\colon \Der^b(A)\to \Der^b(B)$ be an equivalence. Then there is an equivalence of standard type $X^*\otimes_A^{\mathbb{L}}-\colon \Der^b(A)\to \Der^b(B)$ such that 
$F(M^*)\cong X^*\otimes_A^{\mathbb{L}} M^*$ for every $M^*\in \Der^b(A)$. Its quasi-inverse is given by the equivalence of standard type $Y^*\otimes_B^{\mathbb{L}}-$, where $Y^*=\mathbb{R}\Hom_A(X^*,A)$.
\end{thm}

This theorem allows us to always work with equivalences of standard type instead of general derived equivalences. Rickard also proved that such an equivalence gives a standard equivalence of the corresponding enveloping algebras, and hence an isomorphism of Hochschild cohomology rings:

\begin{thm}
\label{thm:hh-derived-equivalence}
Let $A$ and $B$ be $k$-algebras. Suppose there is a derived
equivalence of standard type $X^*\otimes^\mathbb{L}_A-\colon
\Der^b(A)\to \Der^b(B)$, and let $Y^*\otimes^\mathbb{L}_B-$ be an
inverse equivalence. Then we have the following.
\begin{enumerate}
\item The functor $Y^*\otimes^{\mathbb{L}}_B (- \otimes^{\mathbb{L}}_B
X^*)\colon \Der^b(B^e)\to \Der^b(A^e)$ is a derived equivalence of
standard type, and there is an isomorphism $\psi \colon
Y^*\otimes^{\mathbb{L}}_B X^* \to A$ in $\Der^b(A^e)$.
\item Let $\psi_* \colon \End_{\Der^b(A^e)}^*(Y^*
\otimes^{\mathbb{L}}_B X^*) \to \HH^*(A)$, given by $\eta \mapsto
\psi\eta\psi^{-1}$, be the isomorphism of endomorphism rings induced
by $\psi$.  Then the map 
\[\psi_*\circ (Y^*\otimes^{\mathbb{L}}_B
(-\otimes^{\mathbb{L}}_BX^*))\colon \HH^*(B)\to \HH^*(A)\]
is an
isomorphism of graded $k$-algebras.
\end{enumerate}
\end{thm}

The following lemma states that when we are working with finite-dimensional $k$-algebras, the derived tensor product $\otimes^{\mathbb{L}}$ is associative (up to natural isomorphism).

\begin{lem}
\label{lem:derived-tensor-associative}
Let $A$, $B$, $C$ and $D$ be $k$-algebras, and let
${}_A L^*_B$, ${}_B M^*_C$ and ${}_C N^*_D$ be bounded complexes of bimodules. Then there is an isomorphism $(L^* \otimes^{\mathbb{L}}_B M^*)
\otimes^{\mathbb{L}}_C N^* \cong L^* \otimes^{\mathbb{L}}_B (M^*
\otimes^{\mathbb{L}}_C N^*)$ which is natural in all three factors.
More precisely, the following diagram of functors commutes up to natural isomorphism:
\[
\begin{tikzcd}[column sep=7em]
\Der^{-}(A\otimes_k C^{\op})
\ar{dr}{-\otimes^{\mathbb{L}}_C N^*}
&
\Der^{-}(B\otimes_k C^{\op})
\arrow{l}[swap]{L^* \otimes^{\mathbb{L}}_B-}
\arrow{r}{-\otimes^{\mathbb{L}}_C N^*}
&
\Der^{-}(B\otimes_k D^{\op})
\arrow{dl}[swap]{L^* \otimes^{\mathbb{L}}_B-}
\\
\Der^{-}(A\otimes_k B^{\op})
\arrow{u}{-\otimes^{\mathbb{L}}_B M^*}
\arrow{r}[swap]{-\otimes^{\mathbb{L}}_B(M^*\otimes^{\mathbb{L}}_C N^*)}
&
\Der^{-}(A\otimes_k D^{op})
&
\Der^{-}(C\otimes_k D^{\op})
\arrow{u}[swap]{M^* \otimes^{\mathbb{L}}_C-}
\arrow{l}{(L^* \otimes^{\mathbb{L}}_BM^*)\otimes^{\mathbb{L}}_C-}
\end{tikzcd}
\]
\end{lem}
\begin{proof}
First recall that $-\otimes^{\mathbb{L}}_B-\colon \Der^-(A\otimes_k B^{\op})\times \Der^-(B\otimes_k C^{\op})\to \Der^-(A\otimes_k C^{\op})$ is a bifunctor. This is noted as Exercise~10.6.2 in \cite{Wei94}. For a sketch of proof in a slightly different situation we refer the reader to \cite[Appendix~B]{HTT08}. Using this fact we can replace every module by its projective bimodule resolution:

Consider the derived tensor functor $L^* \otimes^{\mathbb{L}}_B -$.
Let $P_{L^*}$ be a projective bimodule resolution of $L^*$.  Since there is
a quasi-isomorphism $P_{L^*} \to L^*$ and $- \otimes^{\mathbb{L}}_B -$ is a
bifunctor, we have a natural isomorphism $L^* \otimes^{\mathbb{L}}_B -
\cong P_{L^*} \otimes^{\mathbb{L}}_B -$. 

Next we claim that if $P^*$ is a bounded above complex of projective $A$-$B$-bimodules, then $P^*\otimes^{\mathbb{L}}_B- = P^*\otimes_B- \colon \Der^-(B\otimes_k C^{\op})\to \Der^-(A\otimes_k C^{\op})$. It suffices to show that the functor $P^*\otimes_B-$ preserves acyclic complexes. Let $Q^*$ be an acyclic complex of $B$-$C$-bimodules. Using the duality $\Hom_{k}(-,k)$ and the $\Hom$-tensor adjunction we obtain an isomorphism
\[
\Homcom_{k}(P^*\otimes_BQ^*, k)\cong \Homcom_{B^{\op}}(P^*, \Hom_{k}(Q^*,k)),
\]
where $\Homcom$ denotes the $\Hom$ complex: Recall that for two complexes $(M^*,d_M)$, $(N^*,d_N)$ of $A$-modules, $\Homcom(M^*,N^*)$ is the complex with $n$-th component $\prod_{i\in \mathbb{Z}}\Hom(M^i, N^{i+n})$ and differential given by $\partial f=d_N\circ f-(-1)^nf\circ d_M$ for a homogeneous $f$ of degree $n$.
Taking homology it follows that
\begin{align*}
H_n\big(\Homcom^{*}_{k}(P^*\otimes_BQ^*, k)\big) &=
H_n\big(\Homcom_{B^{\op}}(P^*, \Hom_{k}(Q^*,k))\big) \\
&\cong \Hom_{\Kom^{-}(B^{\op})}(P^*, \Hom_{k}(Q^*,k)[n])=0
\end{align*}
for any $n$.  Note that the last Hom space is zero because $P^*$ is a complex of projectives and $\Hom_{k}(Q^*,k)[n]$ is an acyclic complex.
This implies that the complex $P^*\otimes_BQ^*$ is acyclic.  This finishes the proof of the claim. By combining the two steps, we get a natural isomorphism $L^* \otimes^{\mathbb{L}}_B - \cong
P_{L^*} \otimes_B -$.

We can describe the other functors similarly.  For the functor $(L^*
\otimes^{\mathbb{L}}_BM^*)\otimes^{\mathbb{L}}_C-$, we need a
projective bimodule resolution of $L^* \otimes^{\mathbb{L}}_BM^*$.  We
claim that $P_{L^*} \otimes_B P_{M^*}$ is such a resolution, where
$P_{L^*}$ and $P_{M^*}$ are projective bimodule resolutions of $L^*$
and $M^*$, respectively.  Using the above description of the functor
$L^* \otimes^{\mathbb{L}}_B -$, we have $L^* \otimes^{\mathbb{L}}_B
M^* \cong P_{L^*} \otimes_B M^* \cong P_{L^*} \otimes_B P_{M^*}$.
Since $k$ is a field, every tensor product of projective bimodules is
a projective bimodule.  This means that the complex $P_{L^*} \otimes_B
P_{M^*}$ is a projective bimodule resolution of $L^*\otimes_B^{\mathbb{L}}M^*$.

Now we have shown that we can replace all the complexes in the
diagrams by their projective resolutions and all the derived tensor
functors by ordinary tensor functors, which are associative.
\end{proof}

\section{Support varieties}\label{supportvarieties}

In this section we recall the definition of support varieties of
bounded complexes over a finite dimensional $k$-algebra $A$ using Hochschild cohomology, and we prove the main result as stated in the introduction. 
In the theory of support varieties, it is standard to work with
classical varieties instead of schemes.  Henceforth, we therefore
assume our ground field $k$ to be algebraically closed.

We use a description of support varieties in terms of the derived
category.  This description comes from the as yet unpublished
paper~\cite{BKSS}.  A short summary can be found
in~\cite[Section~10]{Sol06}, from which we recall the definitions and
results we need.

To define support varieties we need to recall the action of $\HH^*(A)$ on $\Hom_{\Der^b(A)}(M^*,N^*)$ for complexes $M^*$ and $N^*$ in $\Der^b(A)$. Define $\varphi_{M^*}:\HH^*(A)\to \End_{\Der^b(A)}^*(M^*)$ as $\varphi_{M^*} = -\otimes^{\mathbb{L}}_AM^*$. This is a homomorphism of graded rings and the left action of $\End_{\Der^b(A)}^*(N^*)$ on $\Hom_{\Der^b(A)}^*(M^*,N^*)$ induces a left action of $\HH^*(A)$ on this space. In order for an algebra $A$ to have a well-behaved support variety theory, the following two conditions need to be satisfied:

\begin{defn}
Let $A$ be a $k$-algebra. Let $H\subseteq \HH^*(A)$ be a subalgebra with $H^0=\HH^0(A)$. Then $A$ is said to \emphbf{satisfy (Fg)} with respect to $H$ if the following two conditions hold:
\begin{enumerate}[{(Fg}1{)}]
\item The algebra $H$ is a noetherian ring.
\item $\Hom_{\Der^b(A)}^*(M^*, N^*)$ is a finitely generated $H$-module for all bounded complexes $M^*$ and $N^*$ of $A$-modules.
\end{enumerate} 
We say that $A$ \emphbf{satisfies (Fg)} if there exists $H$ such that $A$ satisfies (Fg) with respect to $H$.
\end{defn}

There are several different ways to describe the (Fg) condition.  For
an algebra $A$, the following statements are equivalent:
\begin{enumerate}
\item $A$ satisfies (Fg).
\item $A$ satisfies (Fg) with respect to some commutative subalgebra
$H$ of $\HH^*(A)$.
\item $A$ satisfies (Fg) with respect to $\HH^*(A)$.
\item $A$ satisfies (Fg) with respect to $\HH^{\ev}(A)$.
\end{enumerate}
The equivalence of these statements was shown in~\cite[Propositions 5.5--5.7]{Sol06} under the assumption that $H$ is commutative. Note that the proofs therein do not rely on this assumption. Hence, the statements (1)--(4) are equivalent. Furthermore, under the presence of (Fg1) it suffices to check (Fg2)
for the stalk complex $M^*=N^*=A/\rad{A}$, see
\cite[Proposition~10.3]{Sol06}.

\begin{defn}
Let $A$ be a $k$-algebra which satisfies (Fg), and let $H \subseteq
\HH^*(A)$ be a commutative subalgebra of the Hochschild cohomology ring of
$A$ such that $A$ satisfies (Fg) with respect to $H$.  For a pair
$(M^*, N^*)$ of bounded complexes of $A$-modules, the \emphbf{support
  variety} $\mathcal{V}_A^H(M^*, N^*)$ is the maximal ideal spectrum
$\maxSpec(H/\Ann_H(\Hom^*_{\Der^b(A)}(M^*,N^*)))$. Denote
$\mathcal{V}_A^H(M^*,M^*)$ by $\mathcal{V}_A^H(M^*)$.
\end{defn}

If an algebra $A$ satisfies the (Fg) condition with respect to several
different commutative subalgebras $H \subseteq \HH^*(A)$ of the
Hochschild cohomology ring, then the support variety theory for $A$
may depend on the choice of $H$, as the following example illustrates.

\begin{ex}
Let $A=\mathbb{C}[x,y]/(x^2,y^2)$. According to \cite[Proposition 9.1]{ES11b} this algebra satisfies (Fg) (with respect to $H':=\HH^{2*}(A)$). Its Hochschild cohomology ring can be obtained as follows: In \cite{Hol00}, Holm computed the Hochschild cohomology ring of $\mathbb{C}[X]/(X^2)$. He obtained that $\HH^*(\mathbb{C}[X]/X^2)\cong \mathbb{C}[s,t,u]/(s^2, t^2, 2su, ut)$ with $\deg s=0$, $\deg t=1$ and $\deg u=2$. The Hochschild cohomology ring of $A$ is the tensor product of two copies of the Hochschild cohomology ring of $\mathbb{C}[X]/(X^2)$, see e.g. \cite[Corollary 4.8]{BO08}, i.e. $\HH^{2*}(A)\cong \mathbb{C}[s,u,\tilde{s},\tilde{u}]/(s^2, \tilde{s}^2, 2su, 2\tilde{s}\tilde{u})$. Therefore, $\mathcal{V}^{H'}_A(\mathbb{C})\cong \maxSpec \HH^{2*}(A)\cong \mathbb{A}^2$, the complex plane. On the other hand, there is a $\mathbb{Z}/(2)$-action on $\HH^{2*}(A)$ sending $u$ to $-u$ and $\tilde{u}$ to $-\tilde{u}$. The ring of invariants is $H=\mathbb{C}[s,u^2, \tilde{s}, \tilde{u}^2,u\tilde{u}]\subset H'$. The algebra $A$ also satisfies (Fg) with respect to $H$: It is Noetherian and if $I$ is a finite generating set of $\End^*_{\Der^b(A)}(M,M)$ for $\HH^{2*}(A)$, then $I\cup uI\cup \tilde{u}I$ is a finite generating set of $\End^*_{\Der^b(A)}(M,M)$ for $H$. As $H=\HH^{2*}(A)^{\mathbb{Z}/(2)}$ is the ring of invariants, $\mathcal{V}^{H}_A(\mathbb{C})\cong \maxSpec H\cong \mathbb{A}^2/(\mathbb{Z}/(2))$ is a Kleinian singularity. In particular, it is not isomorphic to $\mathbb{A}^2$ as a variety.   
\end{ex}

The definition of support varieties relies on the map $\varphi_{M^*}
\colon \HH^*(A) \to \End_{\Der^b(A)}(M^*)$ for every bounded complex
$M^*$.  In \cite[Subsection~10.1]{Sol06}, Solberg gives an explicit
description of this map instead of using derived functors.  He then
remarks (in Subsection~10.2) that this explicit description coincides
with the definition of $\varphi_{M^*}$ we stated.  There is, however,
no proof of this fact in~\cite[Subsection~10.2]{Sol06}.  In the following remark, we
recall Solberg's explicit description and show that both approaches
define the same map.  For our purposes, it is most convenient to
consider $\varphi_{M^*}$ in terms of the derived tensor functor.  For
performing concrete computations, however, the explicit description
would be more convenient.

\begin{rmk}
\label{rmk:phi}
For giving the explicit description from \cite{Sol06} we need a projective resolution $P^*\colon \cdots \to P_1 \to P_0
\to 0$ of $A$ as an $A^e$-module, together with a quasi-isomorphism
$\varepsilon \colon P^* \to A$.  Consider a homogeneous element of
degree $n$ in $\HH^*(A)$. Such an element can be represented by a
roof
\[
\roof{A}{\varepsilon}{P^*}{\eta}{A[n]}
\]
with the projective resolution $P^*$ on top and the quasi-isomorphism
$\varepsilon$ as the left map.  To get the result of the explicit description we apply $- \otimes_A M^*$ to this roof, and identify $A\otimes_A M^*$
with $M^*$. We obtain the 
the following roof:
\[
\begin{tikzcd}[row sep=0pt]
&& P^* \otimes_A M^* \ar{ld}[swap]{\varepsilon \otimes M^*} \ar{rd}{\eta \otimes M^*} \\
& M^* \cong A \otimes_A M^* &&
A[n] \otimes_A M^* \cong M^*[n]
\end{tikzcd}
\]

This definition relies on the fact that the map $\varepsilon \otimes
M^*$ is a quasi-isomorphism.  We show this now.
Since $k$ is a field, projective $A$-bimodules are projective as one
sided $A$-modules. This implies that all the kernels of the sequence
$\cdots \to P_1 \to P_0 \to A\to 0$ are projective as right
$A$-modules. These are the same modules as the boundaries of the
complex $P^*$. The assumptions of the K\"unneth formula
\cite[Theorem~10.81]{Rot09} are therefore satisfied for the
complex $P^* \otimes_A M^*$. They are also satisfied for the complex
$A \otimes_A M^*$. From this we obtain the vertical isomorphisms in
the following commutative diagram.
\[
\begin{tikzcd}
H_n(P^* \otimes_A M^*) \arrow{r}{H_n(\varepsilon \otimes M^*)} &
H_n(A \otimes_A M^*) \\
\bigoplus\limits_{p+q=n}H_p(P^*)\otimes_A H_q(M^*)
\arrow{u}{\cong} \arrow{r}{\cong} &
\bigoplus\limits_{p+q=n}H_p(A)\otimes_A H_q(M^*)
\arrow{u}{\cong}
\end{tikzcd}
\]
This shows that the map $\varepsilon \otimes M^*$ is a quasi-isomorphism.

The following commutative diagram, where the bottom horizontal arrow is given by the explicit description, shows that the explicit description coincides with the description in terms of the derived tensor product $\varphi_{M^*}=-\otimes^{\mathbb{L}}M^*$:
\[
\begin{tikzcd}
\End^*_{\Kom^{-,b}(\proj A^e)}(P^*)\arrow{d}{N_{A^e}}\arrow{r}{-\otimes_A M} &\End_{\Kom^{-,b}(\proj A)}^*(P^*\otimes_A M)\arrow{d}{N_A}\\
\End^*_{\Der^b(A^e)}(A)\arrow{r} &\End_{\Der^b(A)}^*(M)
\end{tikzcd}
\] 
\end{rmk}

Now we are ready to prove our main result.

\begin{thm}
\label{thm:main}
Let $A$ and $B$ be derived equivalent $k$-algebras, and let ${}_B
X^*_A$ be a bimodule complex inducing a derived equivalence of
standard type $X^* \otimes^{\mathbb{L}}_A - \colon \Der^b(A) \to
\Der^b(B)$ with quasi-inverse given by the bimodule complex ${}_A
Y^*_B$.  Let $f$ be the inverse of the isomorphism $\psi_*\circ
(Y^*\otimes^{\mathbb{L}}_B (-\otimes^{\mathbb{L}}_BX^*))\colon
\HH^*(B)\to \HH^*(A)$ given in
Theorem~\ref{thm:hh-derived-equivalence}.  Then the following
statements hold:
\begin{enumerate}
\item For any bounded complex $S^*$ of $A$-modules, there is a commutative diagram
\[
\begin{tikzcd}[column sep=huge]
\HH^*(A)
\arrow{r}{\varphi_{S^*}}
\arrow{d}{f}[swap]{\cong}
&
\End_{\Der^b(A)}^*(S^*)
\arrow{d}{\cong}[swap]{X^* \otimes^{\mathbb{L}}_A -}
\\
\HH^*(B)
\arrow{r}{\varphi_{X^* \otimes^{\mathbb{L}}_A S^*}}
&
\End_{\Der^b(B)}^*(X^* \otimes^{\mathbb{L}}_A S^*)
\end{tikzcd}
\]
of graded rings, where the vertical maps are isomorphisms.
\item Let $H$ be a subalgebra of $\HH^*(A)$. Let $H':= f(H)$. Then $A$
satisfies (Fg) with respect to $H$ if and only if $B$ satisfies (Fg)
with respect to $H'$. In particular, $A$ satisfies (Fg) if and only if
$B$ satisfies (Fg).
\item Assume additionally that $H$ is commutative. Then, for bounded complexes $M^*$ and $N^*$ of $A$-modules, there is an isomorphism of varieties:
\[\mathcal{V}_A^H(M^*,N^*)\cong \mathcal{V}_B^{H'}(X^*\otimes^{\mathbb{L}}_A M^*, X^*\otimes^{\mathbb{L}}_A N^*)\]
\end{enumerate}
\end{thm}

\begin{proof}
We first prove part~(1) by constructing the following
commutative diagram of graded rings:
\[
\begin{tikzcd}[column sep=huge]
\HH^*(A)
\arrow{r}{\varphi_{S^*}}
\arrow [rounded corners, to path={ -- ([xshift=-9ex]\tikztostart.west) -- ([xshift=-9ex]\tikztotarget.west) \tikztonodes -- (\tikztotarget)}]{dd}[swap]{f} &
\End_{\Der^b(A)}^*(S^*)
\arrow [rounded corners, to path={ -- ([xshift=12.63ex]\tikztostart.east) -- ([xshift=9ex]\tikztotarget.east) \tikztonodes -- (\tikztotarget)}]{dd}{X^* \otimes^{\mathbb{L}}_A -}
\\
\End_{\Der^b(A^e)}^*(Y^* \otimes^{\mathbb{L}}_B X^*)
\arrow{r}{-\otimes^\mathbb{L}_A S^*}
\arrow{u}{\cong}[swap]{\psi_*}
&
\End_{\Der^b(A)}^*(Y^* \otimes^{\mathbb{L}}_B X^* \otimes^{\mathbb{L}}_A S^*)
\arrow{u}[swap]{(\psi \otimes^{\mathbb{L}}_A S^*)_*}{\cong}
\\
\HH^*(B)
\arrow{r}[swap]{\varphi_{X^* \otimes^{\mathbb{L}}_A S^*}}
\arrow{u}{\cong}[swap]{Y^* \otimes^{\mathbb{L}}_B (- \otimes^{\mathbb{L}}_B X^*)}
&
\End_{\Der^b(B)}^*(X^* \otimes^{\mathbb{L}}_A S^*)
\arrow{u}[swap]{Y^* \otimes^{\mathbb{L}}_B -}{\cong}
\end{tikzcd}
\]
The associativity of the tensor product
(Lemma~\ref{lem:derived-tensor-associative}) yields the commutativity
of the lower part of the diagram.  The upper square commutes since for $\eta\in \End^*_{\Der^b(A)}(Y^*\otimes^{\mathbb{L}}_B X^*)$ there is the following chain of equalities:
\[
(\varphi_{S^*} \psi_*)(\eta)
= (\psi \eta \psi^{-1}) \otimes^{\mathbb{L}}_A S^*
= (\psi \otimes^{\mathbb{L}}_A S^*) (\eta \otimes^{\mathbb{L}}_A S^*) (\psi \otimes^{\mathbb{L}}_A S^*)^{-1}
= (\psi \otimes^{\mathbb{L}}_A S^*)_* ((- \otimes^{\mathbb{L}}_A S^*)(\eta)).
\]
To see that the triangle on the right part of the diagram commutes, start with a map
$\alpha \in \End^*_{\Der^b(A)}(S^*)$.  Applying the maps $X^* \otimes^{\mathbb{L}}_A -$
and $Y^* \otimes^{\mathbb{L}}_B -$ gives the map
$Y^* \otimes^{\mathbb{L}}_B X^* \otimes^{\mathbb{L}}_A \alpha$,
and then by applying the map $(\psi \otimes^{\mathbb{L}}_A S^*)_*$,
we get back to the original map $\alpha$.

We now prove parts (2) and (3).  Consider the following diagram:
\tikzexternalenable
\tikzsetnextfilename{diagram-mainthm}
\[
\begin{tikzpicture}
\node {
\begin{tikzcd}[column sep=huge]
H\otimes_k \Hom^*_{\Der^b(A)}(M^*,N^*)
\arrow{r}{f \otimes_k (X^* \otimes^{\mathbb{L}}_A -)}
\arrow{d}{\varphi_{N^*} \otimes_k \id} &
H'\otimes_k \Hom^*_{\Der^b(B)}(X^*\otimes^{\mathbb{L}}_A M^*, X^*\otimes^{\mathbb{L}}_A N^*)
\arrow{d}{\varphi_{(X^* \otimes^{\mathbb{L}}_A N^*)} \otimes_k \id} \\
\framebox[0.95\width]{$\begin{array}{c}
\End^*_{\Der^b(A)}(N^*) \\
\otimes_k \\
\Hom^*_{\Der^b(A)}(M^*,N^*)
\end{array}$}
\arrow{r}{(X^* \otimes^{\mathbb{L}}_A -) \otimes_k (X^* \otimes^{\mathbb{L}}_A -)}
\arrow{d}{\comp} &
\framebox[0.95\width]{$\begin{array}{c}
\End^*_{\Der^b(B)}(X^*\otimes^{\mathbb{L}}_A N^*) \\
\otimes_k \\
\Hom^*_{\Der^b(B)}(X^*\otimes^{\mathbb{L}}_AM^*, X^*\otimes^{\mathbb{L}}_A N^*)
\end{array}$}
\arrow{d}{\comp} \\
\Hom^*_{\Der^b(A)}(M^*,N^*)
\arrow{r}{X^* \otimes^{\mathbb{L}}_A -} &
\Hom^*_{\Der^b(B)}(X^*\otimes^{\mathbb{L}}_A M^*, X^*\otimes^{\mathbb{L}}_A N^*)
\end{tikzcd}
};
\end{tikzpicture}
\]
\tikzexternaldisable
The upper square is commutative by part~(1). The functoriality of $X^*\otimes^{\mathbb{L}}_A-$ implies that the lower square commutes. Since the horizontal arrows are isomorphisms it is straightforward to check that $\Hom^*_{\Der^b(A)}(M^*,N^*)$ is finitely generated as an $H$-module if and only if $\Hom^*_{\Der^b(B)}(X^*\otimes^{\mathbb{L}}_A M^*, X^*\otimes^{\mathbb{L}}_A N^*)$ is finitely generated as an $H'$-module. It is also easy to see that the corresponding annihilators are isomorphic. Part (2) and (3) follow. 
\end{proof}

\begin{rmk}\label{standardnotneeded}
Let $F\colon \Der^b(A)\to \Der^b(B)$ be an arbitrary derived equivalence (not necessarily of standard type). By Theorem \ref{rickardstheorem}, there is some equivalence of standard type 
\[X^*\otimes_{A}^{\mathbb{L}}-\colon \Der^b(A)\to \Der^b(B)\] 
which agrees with $F$ on objects. Hence, under the hypotheses of the foregoing theorem, there are isomorphisms 
\[\mathcal{V}_A^H(M^*,N^*)\cong \mathcal{V}_B^{H'}(X^*\otimes_A^{\mathbb{L}} M^*,X^*\otimes_A^{\mathbb{L}} N^*)\cong \mathcal{V}_B^{H'}(F(M^*),F(N^*)).\]
The subtlety here is that $H'$ is not constructed directly using $F$ but going through determining a standard derived equivalence $X^*\otimes_A^{\mathbb{L}}-$ which agrees with $F$ on objects. This is not necessary if one sets $H=\HH^{\ev}(A)$ and $H'=\HH^{\ev}(B)$. In this case, one has 
\[\mathcal{V}_A^{\HH^{\ev}(A)}(M^*,N^*)\cong \mathcal{V}_B^{\HH^{\ev}(B)}(F(M^*),F(N^*)).\]
\end{rmk}

Summing up, the (Fg) condition on Hochschild cohomology as well as the support variety theory are invariant under derived equivalences (of standard type). In the next section we will provide an instance of this theorem.

\section{How to use the Theorem?}\label{examples}

In this section, we demonstrate how Theorem~\ref{thm:main} can be used
to produce examples of algebras which satisfy the (Fg) condition.  We
first summarize some earlier known results which are useful for
determining whether a given algebra satisfies (Fg).  Then we use some
of these results together with Theorem~\ref{thm:main} to discuss a new
example of an algebra satisfying~(Fg).  Throughout this section, we
let $k$ be an algebraically closed field.

As remarked in the introduction any algebra which satisfies (Fg) is a Gorenstein
algebra (see \cite[Proposition 1.2]{EHSST04}).  Recall that an algebra
$A$ is \emphbf{Gorenstein}, sometimes also called
\emphbf{Iwanaga-Gorenstein}, if $\injdim {}_AA<\infty$ and $\injdim
A_A<\infty$.  For Nakayama algebras, the following result by Nagase
reduces the problem of determining (Fg) to the problem of determining
Gorensteinness.

\begin{thm} \cite[Corollary~10]{Nagase}
\label{thm:nakayama-gorenstein-fg}
Let $A$ be a Nakayama $k$-algebra.  Then $A$ satisfies the (Fg)
condition if and only if $A$ is a Gorenstein algebra.
\end{thm}

Furthermore, there exists a concrete algorithm by Ringel~\cite{Ringel}
for determining whether a Nakayama algebra is Gorenstein.  For
Nakayama algebras with at most three simples, there is also a method
by Chen and Ye~\cite{ChenYe} for determining Gorensteinness.  Thus,
for Nakayama algebras, the problem of determining (Fg) is completely
solved, in the sense that there exists an algorithm for determining
whether any given Nakayama algebra satisfies (Fg).

The following result shows that in certain situations, the problem of
determining whether (Fg) is satisfied for an algebra $A$ can be reduced
to the same problem for the smaller algebra $eAe$, where $e$ is some
idempotent in~$A$.

\begin{thm} \cite[Theorem 8.1 (i),(iv)]{PSS}
\label{thm:eAe}
Let $A$ be a $k$-algebra, let $e$ be an idempotent in~$A$, and let $B
= A/\langle e \rangle$.  Assume that $\projdim_A (B/\rad B) < \infty$
and $\projdim_{(eAe)^{\op}} Ae < \infty$.  Then $A$ satisfies the (Fg)
condition if and only if $eAe$ satisfies the (Fg) condition.
\end{thm}

For simplicity, we only consider derived equivalences induced by
tilting modules in our example.  By a result of Happel~\cite{Hap87},
if $T$ is a tilting module, then $T\otimes^{\mathbb{L}}_A-\colon
\Der^b(A)\to \Der^b(\End_A(T)^{\op})$ is a derived equivalence.  We
recall the definition of a tilting module:

\begin{defn}
\label{deftiltingmod}
Let $A$ be a $k$-algebra, and let $T$ be an
$A$-module.  Consider the following conditions on $T$:
\begin{enumerate}[(i)]
\item $\projdim_A T < \infty$.
\item $\Ext_A^i(T, T) = 0$ for every $i > 0$.
\item There exists an exact sequence $0 \to A \to T_0 \to \cdots \to
T_m \to 0$ of $A$-modules where every $T_i$ is in $\add T$.
\end{enumerate}
If all these conditions are true, then the module $T$ is called a
\emphbf{tilting module}.  If conditions (i) and~(ii) are true, and the
number of nonisomorphic direct summands of $T$ is $n-1$, where $n$ is
the number of simple $A$-modules (up to isomorphism), then $T$ is
called an \emphbf{almost complete tilting module}.  If $T$ is an
almost complete tilting module and $N \in \modu A$ is a module such
that $T \oplus N$ is a tilting module and $\add T \cap \add N = 0$,
then $N$ is called a \emphbf{complement} to $T$.
\end{defn}

In our example, we use the following result to find tilting modules.
The analogue of this result for cotilting modules is stated in \cite[Proposition 3.2]{BS98}.
We include a proof for the sake of completeness.  For the notion of
left approximations, we refer to~\cite{AR92}.

\begin{prop}
\label{prop:tilting}
Let $A$ be a $k$-algebra, let $M$ be an almost complete tilting
$A$-module, and let $X$ be an indecomposable complement to $M$.  If $f
\colon X \to E$ is a map which is both a left $\add M$-approximation
of $X$ and a monomorphism, then $M\oplus \Coker f$ is a
tilting module.
\end{prop}

\begin{proof}
Let $Y = \Coker f$. There is an exact sequence
$\eta\colon 0 \to X \xrightarrow{f} E \to Y \to 0$.
We check that the three requirements of Definition~\ref{deftiltingmod} 
are satisfied for the module $M \oplus Y$.

Since $M \oplus X$ is a tilting module, both $E\in \add M$ and $X$ have finite
projective dimension. It follows that $M\oplus Y$ has finite
projective dimension. 
For the second requirement, we apply the functor $\Hom_A(-,M)$ to the sequence $\eta$ and get a long exact sequence
\begin{align*}
0 \to
&\Hom_A(Y,M) \to
 \Hom_A(E,M) \xrightarrow{f^*}
 \Hom_A(X,M) \to \Ext_A^1(Y,M) \to \cdots 
\end{align*}
Since $M \oplus X$ is a tilting module, $E$ is in $\add
M$ and $f^*$ is an epimorphism, it follows that $\Ext_A^i(Y,M)=0$ for $i>0$.
Similarly, we show that $\Ext_A^i(M,Y)$ and $\Ext_A^i(Y,Y)$ are zero for $i > 0$.
We infer that $\Ext_A^i(M \oplus Y, M \oplus Y) =
0$ for $i > 0$.

We now check the third requirement, i.e. we construct an exact sequence 
\begin{equation}
\tag{$*$}
\label{eqn:tilting-sequence}
 0 \to A \to T'_0 \to \cdots \to T'_{m'} \to 0
\end{equation}
 of $A$-modules, where each $T'_i$ lies in $\add (M \oplus Y)$. Since $M \oplus X$ is a tilting module, there exists an exact sequence $0 \to A \stackrel{\iota_0}{\to} T_0 \to \cdots \to T_m \to 0$ of $A$-modules, where each $T_i$ is in $\add (M \oplus X)$. Taking out all summands
isomorphic to the indecomposable module $X$, we can decompose the
module $T_i$ as $T_i \cong M_i \oplus X^{t_i}$, where $M_i$ is in
$\add M$ and $t_i \ge 0$ is an integer. Define the monomorphism $f_i\colon T_i \to M_i \oplus E^{t_i}$, by using the identity on $M_i$ and the map $f$ on each summand in
$X^{t_i}$. To construct the exact sequence \eqref{eqn:tilting-sequence}, we start
with the exact commutative diagram
\[
\begin{tikzcd}
0 \arrow{r} &
A \arrow[equal]{d} \arrow{r}{\iota_0} &
T_0 \arrow{r} \arrow[rightarrowtail]{d}{f_0} &
K_{1} \arrow{r}\arrow[rightarrowtail]{d}{\alpha_1} &
0 \\
0 \arrow{r} &
A \arrow{r}{f_0\iota_0} & M_0 \oplus E^{t_0} \arrow{r} &
K_{1}^{'} \arrow{r} &
0
\end{tikzcd}
\]
where $\alpha_1\colon K_1\to K_1'$ is induced by the universal property of cokernels. Let $\iota_1\colon K_1\to T_1$ be the map induced by $T_0\to T_1$. Setting $T'_0:=M_0 \oplus E^{t_0}$ in the lower exact sequence of the above diagram, we get the first step of the sequence \eqref{eqn:tilting-sequence}. Let $Y_1=\Coker{\alpha_1}$. Then $Y_1\cong \Coker{f_0}\in \add{Y}$. Taking the pushout of $\alpha_1$ along $f_1\iota_1$, we obtain the exact commutative diagram
\begin{equation}
\label{eqn:tilting-diagram-1}
\begin{tikzcd}
0 \arrow{r} &
K_1 \arrow[rightarrowtail]{d}{f_1\iota_1} \arrow{r}{\alpha_1} &
K_1' \arrow{r}{\beta_1} \arrow[rightarrowtail]{d}{\iota'_1} &
Y_1 \arrow{r}\arrow[equal]{d} &
0 \\
0 \arrow{r} &
M_1\oplus E^{t_1} \arrow{r}{} &
T_1^{'} \arrow{r} &
Y_{1} \arrow{r} & 0
\end{tikzcd}
\end{equation}
We know from above that $\Ext_A^1(Y,M) = 0$. Since $Y_1$ is in
$\add Y$ and $M_1 \oplus E^{t_1}$ is in $\add M$, it follows that the bottom exact
sequence of (\ref{eqn:tilting-diagram-1}) splits, and thus $T'_1$ is in $\add (M \oplus Y)$. Let $K'_2=\Coker{\iota'_1}\cong \Coker{f_1\iota_1}$. Hence, up to this point, we have constructed the following part of the exact sequence \eqref{eqn:tilting-sequence}:
\[
0\to A\to T'_0\to T'_1\to K'_2\to 0
\]
with $T'_0$ and $T'_1$ in $\add(M \oplus Y)$. 
By applying the Snake Lemma to the exact commutative diagram
\[
\begin{tikzcd}
0 \arrow{r} &
K_1 \arrow{r}{f_1\iota_1} \arrow[rightarrowtail]{d}{\iota_1} &
M_1 \oplus E^{t_1} \arrow{r} \arrow[equal]{d} &
K'_{2} \arrow{r} \arrow[twoheadrightarrow]{d}{\beta_2} &
0 \\
0 \arrow{r} &
T_1 \arrow{r}{f_1} &
M_1 \oplus E^{t_1} \arrow{r} &
Y_{2} \arrow{r} &
0
\end{tikzcd}
\]
we get an exact sequence similar to the upper exact sequence of (\ref{eqn:tilting-diagram-1}).
The construction continues in the same way as above and the proof is finished. 
\end{proof}

Using the above results, we now construct two derived equivalent
algebras, where one of the algebras is known to satisfy (Fg).  By
Theorem~\ref{thm:main}, it follows that the other algebra also satisfies
(Fg).

\begin{ex}
\label{ex:nakayama}
Let $A = kQ/\langle\rho\rangle$ be the $k$-algebra given by the
following quiver and relations:
\[
Q \colon
\begin{tikzcd}
{}
& 1 \arrow{dr}{a} \\
3 \arrow{ur}{c} &&
2 \arrow{ll}{b}
\end{tikzcd}
\qquad
\rho = \{ bacba, cbac \}.
\]
Then $A$ is a Nakayama algebra.  The indecomposable projective modules
are
\[
P_1\colon
\begin{array}{c}
1 \\
2 \\
3 \\
1 \\
2
\end{array},
\qquad
P_2\colon
\begin{array}{c}
2 \\
3 \\
1 \\
2 \\
3
\end{array}
\qquad
\text{and}
\qquad
P_3\colon
\begin{array}{c}
3 \\
1 \\
2 \\
3
\end{array}.
\]
By
\cite[Proposition~3.14]{ChenYe}, this means that $A$ is a Gorenstein
algebra (its normalized admissible sequence being $(4, 5, 5)$).  Since $A$ is a Gorenstein Nakayama algebra, it satisfies (Fg)
by Theorem~\ref{thm:nakayama-gorenstein-fg}.

We now find an algebra which is derived equivalent to~$A$.  We use
Proposition~\ref{prop:tilting} to find a tilting module.  Consider the
almost complete tilting module $P_1 \oplus P_2$ with complement~$P_3$.
The inclusion $P_3 \hookrightarrow P_2$ is a left $\add (P_1 \oplus
P_2)$-approximation of $P_3$ and a monomorphism.  By
Proposition~\ref{prop:tilting}, the module $T = P_1 \oplus P_2 \oplus S_2$
is a tilting module.

We find the quiver and relations for the algebra $\End_A(T)$.  The
vertices of the quiver correspond to the three indecomposable summands
$P_1$, $P_2$ and $S_2$ of $T$.  The arrows of the quiver correspond to
maps between these indecomposable modules.  The following diagram 
depicts a $k$-basis for $\End_A(P_1\oplus P_2\oplus S_2)$ (except for the identity
maps):
\begin{equation}
\label{eqn:ex-nakayama-maps}
%
%
\begin{tikzpicture}[xscale=1,yscale=1,baseline=20]
\node (S2) at (0,0) {$S_2$};
\node (P2) at (3.6,0) {$P_2$};
\node (P1) at (1.8,1.6) {$P_1$};
\draw[line width=0.5pt,-cm to] (S2) to node[above left]{\tiny$\theta$} (P1);
\draw[line width=0.5pt,-cm to] (P2) to node[below]{\tiny$\eta$} (S2);
\draw[line width=0.5pt,-cm to,bend right=50] (P2.north) to node[right]{\tiny$\epsilon$} (P1.east);
\draw[line width=0.5pt,-cm to] ([xshift=-2.5pt,yshift=-2.5pt] P1.south east) to node[above right=2pt]{\tiny$\delta$} ([xshift=-2.5pt,yshift=-2.5pt] P2.north west);
\draw[line width=0.5pt,-cm to] ([xshift=2.5pt,yshift=2.5pt] P2.north west) to node[below left=2pt] {\tiny$\gamma$} ([xshift=2.5pt,yshift=2.5pt] P1.south east);
\draw[line width=0.5pt,-cm to,loop right] (P2) to node[right]{\tiny$\beta$} (P2);
\draw[line width=0.5pt,-cm to,loop above] (P1) to node[above]{\tiny$\alpha$} (P1);
\end{tikzpicture}
\end{equation}
The maps are defined as follows:
\[
\alpha \colon
\begin{tikzcd}[row sep=0ex]
1 \arrow[mapsto]{dddr} & 1 \\
2 \arrow[mapsto]{dddr} & 2 \\
3                      & 3 \\
1                      & 1 \\
2                      & 2
\end{tikzcd}
\qquad
\beta \colon
\begin{tikzcd}[row sep=0ex]
2 \arrow[mapsto]{dddr} & 2 \\
3 \arrow[mapsto]{dddr} & 3 \\
1                      & 1 \\
2                      & 2 \\
3                      & 3
\end{tikzcd}
\qquad
\gamma \colon
\begin{tikzcd}[row sep=0ex]
1 \arrow[mapsto]{ddr} & 2 \\
2 \arrow[mapsto]{ddr} & 3 \\
3 \arrow[mapsto]{ddr} & 1 \\
1                     & 2 \\
2                     & 3
\end{tikzcd}
\qquad
\delta \colon
\begin{tikzcd}[row sep=0ex]
2 \arrow[mapsto]{dr} & 1 \\
3 \arrow[mapsto]{dr} & 2 \\
1 \arrow[mapsto]{dr} & 3 \\
2 \arrow[mapsto]{dr} & 1 \\
3                    & 2
\end{tikzcd}
\]
\[
\epsilon \colon
\begin{tikzcd}[row sep=0ex]
2 \arrow[mapsto]{ddddr} & 1 \\
3                       & 2 \\
1                       & 3 \\
2                       & 1 \\
3                       & 2
\end{tikzcd}
\qquad
\eta \colon
\begin{tikzcd}[row sep=0ex]
2 \arrow[mapsto]{r} & 2 \\
3                       \\
1                       \\
2                       \\
3
\end{tikzcd}
\qquad
\theta \colon
\begin{tikzcd}[row sep=0ex]
2 \arrow[mapsto]{ddddr} & 1 \\
                        & 2 \\
                        & 3 \\
                        & 1 \\
                        & 2
\end{tikzcd}
\]
We have the
following relations between these maps:
\begin{equation}
\label{eqn:ex-nakayama-relations}
\epsilon = \alpha\delta = \delta\beta = \theta\eta, \qquad
\alpha = \delta\gamma
\qquad\text{and}\qquad
\beta = \gamma\delta.
\end{equation}
We observe that the maps $\alpha$, $\beta$ and $\epsilon$ can be
expressed in terms of other maps.  By removing these maps from
\eqref{eqn:ex-nakayama-maps}, we obtain the quiver of the algebra
$\End_A(T)$.  We find the relations of the algebra by using
equations~\eqref{eqn:ex-nakayama-relations} and observing that the
compositions $\gamma\delta\gamma$, $\eta\gamma$ and $\gamma\theta$ are
zero.  Let $Q'$ and $\rho'$ denote the quiver and relations
\[
Q'\colon
%
%
\begin{tikzpicture}[xscale=1,yscale=1,baseline=20]
\node (S2) at (0,0) {$\mathrm{III}$};
\node (P2) at (3.3,0) {$\mathrm{II}$};
\node (P1) at (1.65,1.5) {$\mathrm{I}$};
\draw[line width=0.5pt,-cm to] (S2) to node[above left]{\tiny$\theta$} (P1);
\draw[line width=0.5pt,-cm to] (P2) to node[below]{\tiny$\eta$} (S2);
\draw[line width=0.5pt,-cm to] ([xshift=-2.5pt,yshift=-2.5pt] P1.south east) to node[above right=2pt]{\tiny$\delta$} ([xshift=-2.5pt,yshift=-2.5pt] P2.north west);
\draw[line width=0.5pt,-cm to] ([xshift=2.5pt,yshift=2.5pt] P2.north west) to node[below left=2pt] {\tiny$\gamma$} ([xshift=2.5pt,yshift=2.5pt] P1.south east);
\end{tikzpicture}
\qquad\text{and}\qquad
\rho' =
\{ \gamma\delta\gamma, \eta\gamma,
   \delta\gamma\delta - \theta\eta, \gamma\theta \}.
\]
Then the endomorphism algebra $\End_A(T)$ must be a factor algebra of
the algebra $kQ'/\rho'$.  Furthermore, we observe that both these
algebras have $k$-dimension $10$, so $\End_A(T) \cong
kQ'/\rho'$.

Let $B = \End_A(T)^{\op}$.  Then the algebras $A$ and $B$ are
derived equivalent.  Since $A$ satisfies the (Fg) condition,
Theorem~\ref{thm:main} tells us that $B$ also satisfies the (Fg)
condition.

We finally observe that we could not have used Theorem~\ref{thm:eAe}
to find out that the algebra $B$ satisfies (Fg).  We can easily check
that all the three simple $B$-modules have infinite projective
dimension.  This means that for any choice of nontrivial idempotent
$e$ in $B$, the assumption of Theorem~\ref{thm:eAe}
is not satisfied.  Therefore it is not possible to use this theorem to reduce
the question of whether (Fg) holds for $B$ to the same question about
a smaller algebra $eBe$.
\end{ex}

This example demonstrates that Theorem~\ref{thm:main} can be used to
show that (Fg) holds for an algebra which doesn't belong to one of the classes
of algebras where (Fg) is known to hold, e.g. group algebras, and for which other
general theorems for deducing (Fg) do not apply.

\bibliographystyle{alpha}
\bibliography{publication}

\end{document}